\documentclass{article}  

\usepackage{amsmath}
\usepackage{amsfonts}
\usepackage{amsthm}
\usepackage{xspace}
\usepackage{color}
\usepackage{paralist}
\usepackage{boxedminipage}
\usepackage{algorithm,algorithmic}

\usepackage{tikz}
\usetikzlibrary{fadings}
\usetikzlibrary{calc}

\newtheorem{theorem}{Theorem}
\newtheorem{proposition}{Proposition}
\newtheorem{definition}{Definition}
\newtheorem{lemma}{Lemma}

\def\mypar#1{\smallskip\noindent\textbf{#1}}

\newcommand{\RR}{\mathbb R}
\newcommand{\ZZ}{\mathbb Z}

\def\b1{{\bf 1}}
\def\be{{\bf e}}
\def\ba{{\bf a}}

\def\bb{{\bf b}}

\def\bx{{\bf x}}
\def\by{{\bf y}}
\def\bv{{\bf v}}

\def\bz{{\bf z}}
\def\RR{{\mathbb R}}

\def\prob#1{\textup{\textsf{#1}}\xspace}

\def\HyperLabel{\prob{Hypergraph Labeling}}

\def\sep{separating set}

\begin{document}

\title{\Large Sperner's Colorings and \\ Optimal Partitioning of the Simplex\footnote{To appear in {\em ``A Journey through Discrete Mathematics. A Tribute to Ji\v{r}\'{i} Matou\v{s}ek"}, edited by Martin Loebl, Jaroslav Ne\v{s}et\v{r}il and Robin Thomas, due to be published by Springer. The first two results of this paper have also appeared in \cite{MV15}.}}
\author{Maryam Mirzakhani\thanks{Dept. of Mathematics, Stanford University, Stanford, CA; {\tt mmirzakh@stanford.edu}}
 \and Jan Vondr\'ak\thanks{Dept. of Mathematics, Stanford University, Stanford, CA; {\tt jvondrak@stanford.edu}}}
\date{}

\maketitle

%\centerline{\em In memory and dedication to Jirka Matou\v{s}ek}

%\pagenumbering{arabic}
%\setcounter{page}{1}%Leave this line commented out.

\begin{abstract} \small\baselineskip=9pt 
We discuss coloring and partitioning questions related to Sperner's Lemma, originally motivated by an application in hardness of approximation. Informally, we call a partitioning of the $(k-1)$-dimensional simplex into $k$ parts, or a labeling of a lattice inside the simplex by $k$ colors, ``Sperner-admissible" if color $i$ avoids the face opposite to vertex $i$. The questions we study are of the following flavor: What is the Sperner-admissible labeling/partitioning that makes the total area of the boundary between different colors/parts as small as possible?

First, for a natural arrangement of ``cells" in the simplex, we prove an optimal lower bound on the number of cells that must be non-monochromatic in any Sperner-admissible labeling. This lower bound is matched by a simple labeling where each vertex receives the minimum admissible color.

Second, we show for this arrangement that in contrast to Sperner's Lemma, there is a Sperner-admissible labeling such that every cell contains at most $4$ colors. 
%We present an interpretation of this statement in the context of fair division: There is a preference function on $\Delta_{k,q} = \{ \bx \in \RR_+^k: \sum_{i=1}^{k} x_i = q\}$ such that for any division of $q$ units of a resource, $(x_1, x_2, \ldots, x_k) \in \Delta_{k,q}$ such that $\sum_{i=1}^{k} \lfloor x_i \rfloor = q-1$, at most $4$ players out of $k$ are satisfied.

%Our first result is a solution to a discrete variant of this question.
%Let $V_{k,q} = \{ \bv \in \ZZ_+^k: \sum_{i=1}^{k} v_i = q \}$ be a lattice inside the simplex and $E_{k,q} = \{ \{\ba+\be_1, \ba+\be_2, \ldots, \ba+\be_k\}: \ba \in \ZZ_+^k, \sum_{i=1}^{k} a_i = q-1 \}$ a set of ``cells" on this lattice. Call a labeling $\ell:V_{k,q} \rightarrow [k]$ ``Sperner-admissible" if $v_{\ell(\bv)} >  0$ for every $\bv \in V_{k,q}$. Then for every Sperner-admissible labeling there are at least ${q+k-3 \choose k-2}$ non-monochromatic cells in $E_{k,q}$ (and this is tight for $\ell(\bv) = \min \{j: v_j>0\}$).
%This implies an optimal Unique-Games hardness of approximation for the ``Hypergraph Labeling" problem, which was the original motivation for this question.
%the \HMALC problem \cite{ChekuriE11a}: Given a $k$-uniform hypergraph $H = (V,E)$ with color lists $L(v) \subseteq [k] \ \forall v \in V$, find a labeling $\ell(v) \in L(v)$ that minimizes the number of non-monochromatic hyperedges. We also show that a $(k-1)$-approximation can be achieved. \\

Finally, we prove a geometric variant of the first result: For any Sperner-admissible partition of the regular simplex, the total surface area of the boundary shared by at least two different parts is minimized by the Voronoi partition $(A^*_1,\ldots,A^*_k)$ where $A^*_i$ contains all the points whose closest vertex is $i$. We also discuss possible extensions of this result to general polytopes and some open questions.
\end{abstract}

\section{Introduction}

Sperner's Lemma is a gem in combinatorics which was originally discovered by Emmanuel Sperner \cite{Sperner28} as a tool to derive a simple proof of Brouwer's Fixed Point Theorem. Since then, Sperner's Lemma has seen numerous applications, notably in the proof of existence of mixed Nash equilibria \cite{Nash51}, in fair division \cite{Su99}, and recently it played an important role in the study of computational complexity of finding a Nash equilibrium \cite{DaskalakisGP09,ChenDT09}. At a high level, Sperner's Lemma states that for any coloring of a simplicial subdivision of a simplex satisfying certain boundary conditions, there must be a ``rainbow cell" that receives all possible colors. We review Sperner's Lemma in Section~\ref{sec:Sperner}.

The starting point of this work was a question that arises in the study of approximation algorithms for a certain hypergraph labeling problem \cite{EV14}. The question posed by \cite{EV14}, while in some ways reminiscent of Sperner's Lemma, is different in the following sense: Instead of asking whether there exists a rainbow cell for any admissible coloring, the question is what is the minimum possible number of cells that must be {\em non-monochromatic}. (Also, the question arises for a particular regular lattice inside the simplex rather than an arbitrary subdivision.) In this paper, we resolve this question and investigate some related problems.

Before we state our results, let us note the following connection. As the granularity of the subdivision tends to zero, Sperner's Lemma becomes a statement about certain geometric partitions of the simplex: for any Sperner-admissible partition, where part $i$ avoids the face opposite to vertex $i$, there must be a point where all parts meet. This result is known as the Knaster-Kuratowski-Mazurkiewicz Lemma \cite{KKM29}.
%(And this statement also implies Brouwer's Fixed Point Theorem.)
In contrast, the questions we are studying are concerned with the measure of the boundary where at least two different parts meet: This can be viewed as a multi-colored isoperimetric inequality, where we try to partition the simplex in a certain way, so that the surface area of the union of all pairwise boundaries (what we call a {\em \sep}) is minimized. The way we measure the \sep~also affects the problem; the discrete version of the question that is of primary interest to us is mandated by the application in \cite{EV14}. In the geometric setting, a natural notion of surface area is the Minkowski content of the \sep~(which coincides with other notions of volume for well-behaved sets). We give an optimal answer to this question for a regular simplex and discuss other related questions.

To state our results formally, we need some notation that we introduce in Section~\ref{sec:prelims}. We postpone our contributions to Sections~\ref{sec:lattice-coloring}---\ref{sec:simplex-partition}, after a discussion of Sperner's Lemma in Section~\ref{sec:Sperner}.

\section{Preliminaries}
\label{sec:prelims}

We denote vectors in boldface, such as $\bv \in \RR^k$. The coordinates of $\bv$ are written in italics,
such as $\bv = (v_1,\ldots,v_k)$. By $\be_i$, we denote the canonical basis vectors $(0,\ldots,1, \ldots, 0)$.
By $\mbox{conv}(\bv_1,\ldots,\bv_k)$, we denote the convex hull of the respective vectors.

\subsection{Simplicial subdivisions of the simplex}
\label{sec:sperner-simplex}

Consider the $(k-1)$-dimensional simplex defined by 
	$$\Delta_{k} = \mbox{conv}(\be_1,\ldots,\be_k)
	 = \left\{ \bx = (x_1, x_2,\ldots, x_k) \in \RR^k:
	\bx \geq 0,	\sum_{i=1}^{k} x_i = 1 \right\}.$$

\mypar{Simplicial subdivision.}
%\begin{definition}
A simplicial subdivision of $\Delta_{k}$ is a collection of simplices (``cells") $\Sigma$ such that
\begin{compactitem}
\item The union of the cells in $\Sigma$ is the simplex $\Delta_{k}$.
\item For any two cells $\sigma_1, \sigma_2 \in \Sigma$, their intersection is either empty
 or a full face of a certain dimension shared by $\sigma_1, \sigma_2$.
\end{compactitem}
%\end{definition}
%We describe a concrete subdivision of $\Delta_{k}$ in Section~\ref{sec:HJALC}.

\medskip

\mypar{The Simplex-Lattice Hypergraph.}
Next, we describe a specific configuration of cells in a simplex; this configuration is actually not a full subdivision
since its cells do not cover the full volume of the simplex. It can be completed to a subdivision if desired.\footnote{This
specific configuration arises in \cite{EV14} as an integrality gap example for a certain hypergraph labeling problem;
see also \cite{MV15} for more details.}

Let $q \geq 1$ be an integer and define
	$$\Delta_{k,q} = \left\{ \bx = (x_1, x_2,\ldots, x_k) \in \RR^k:
	\bx \geq 0,	\sum_{i=1}^{k} x_i = q \right\}.$$

We consider a vertex set of all the points in $\Delta_{k,q}$ with integer
coordinates:
	$$ V_{k,q} = \left\{ \ba = (a_1, a_2, \ldots, a_k) \in \ZZ^k:
	\ba \geq 0, \sum_{i=1}^{k} a_i = q \right\}.$$

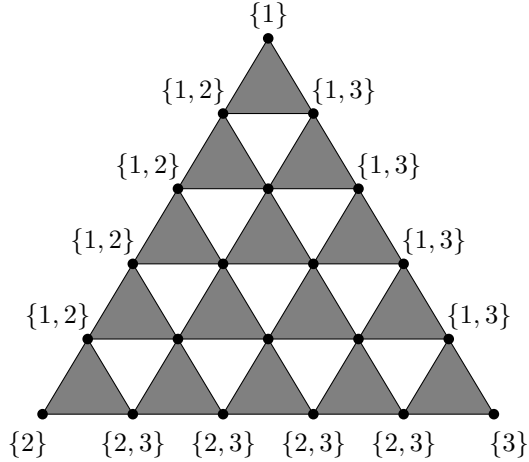
\begin{figure}[h]
\caption{\small The Simplex Lattice Hypergraph for $k=3$ and $q=5$, with hyperedges shaded in gray.
The gray triangles together with the white triangles form a simplicial subdivision.
The lists of admissible colors are given on the boundary; for internal vertices the lists are all $\{1,2,3\}$.
 }

\label{fig:Sperner-setup}

\begin{tikzpicture}

%\begin{scope}[shift={(-1,0)},transform canvas={scale=1.0}]

\draw (-1,0) node {  };

%\draw [color=gray] (5,6.2) -- (0.5,-0.5) -- (9.5,-0.5) -- (5,6.2);
%\draw (1.0,-0.2) node {$V$};

\filldraw [fill=gray] (2,0) -- (3.2,0) -- (2.6,1) -- cycle;
\filldraw [fill=gray] (3.2,0) -- (4.4,0) -- (3.8,1) -- cycle;
\filldraw [fill=gray] (4.4,0) -- (5.6,0) -- (5,1) -- cycle;
\filldraw [fill=gray] (5.6,0) -- (6.8,0) -- (6.2,1) -- cycle;
\filldraw [fill=gray] (6.8,0) -- (8,0) -- (7.4,1) -- cycle;

\filldraw [fill=gray] (2.6,1) -- (3.8,1) -- (3.2,2) -- cycle;
\filldraw [fill=gray] (3.8,1) -- (5,1) -- (4.4,2) -- cycle;
\filldraw [fill=gray] (5,1) -- (6.2,1) -- (5.6,2) -- cycle;
\filldraw [fill=gray] (6.2,1) -- (7.4,1) -- (6.8,2) -- cycle;

\filldraw [fill=gray] (3.2,2) -- (4.4,2) -- (3.8,3) -- cycle;
\filldraw [fill=gray] (4.4,2) -- (5.6,2) -- (5,3) -- cycle;
\filldraw [fill=gray] (5.6,2) -- (6.8,2) -- (6.2,3) -- cycle;

\filldraw [fill=gray] (3.8,3) -- (5,3) -- (4.4,4) -- cycle;
\filldraw [fill=gray] (5,3) -- (6.2,3) -- (5.6,4) -- cycle;

\filldraw [fill=gray] (4.4,4) -- (5.6,4) -- (5,5) -- cycle;

\fill (2,0) circle (2pt);
\fill (3.2,0) circle (2pt);
\fill (4.4,0) circle (2pt);
\fill (5.6,0) circle (2pt);
\fill (6.8,0) circle (2pt);
\fill (8,0) circle (2pt);

\fill (2.6,1) circle (2pt);
\fill (3.8,1) circle (2pt);
\fill (5,1) circle (2pt);
\fill (6.2,1) circle (2pt);
\fill (7.4,1) circle (2pt);

\fill (3.2,2) circle (2pt);
\fill (4.4,2) circle (2pt);
\fill (5.6,2) circle (2pt);
\fill (6.8,2) circle (2pt);

\fill (3.8,3) circle (2pt);
\fill (5,3) circle (2pt);
\fill (6.2,3) circle (2pt);

\fill (4.4,4) circle (2pt);
\fill (5.6,4) circle (2pt);

\fill (5,5) circle (2pt);

\draw (5,5.3) node {$\{1\}$};
\draw (1.8,-0.4) node {$\{2\}$};
\draw (8.2,-0.4) node {$\{3\}$};

\draw (4.0,4.3) node {$\{1,2\}$};
\draw (3.4,3.3) node {$\{1,2\}$};
\draw (2.8,2.3) node {$\{1,2\}$};
\draw (2.2,1.3) node {$\{1,2\}$};

\draw (6.0,4.3) node {$\{1,3\}$};
\draw (6.6,3.3) node {$\{1,3\}$};
\draw (7.2,2.3) node {$\{1,3\}$};
\draw (7.8,1.3) node {$\{1,3\}$};

\draw (3.2,-0.4) node {$\{2,3\}$};
\draw (4.4,-0.4) node {$\{2,3\}$};
\draw (5.6,-0.4) node {$\{2,3\}$};
\draw (6.8,-0.4) node {$\{2,3\}$};

%\draw (5,3.3) node {$1$};
%\draw (4.4,2.3) node {$2$};
%\draw (5.6,2.3) node {$3$};
%\draw (3.8,1.3) node {$2$};
%\draw (5,1.3) node {$2$};
%\draw (6.2,1.3) node {$3$};

%\end{scope}

\end{tikzpicture}

\end{figure}

The {\em Simplex-Lattice Hypergraph} is a $k$-uniform hypergraph $H_{k,q} = (V_{k,q},E_{k,q})$
whose hyperedges (which we also call {\em cells} due to their
geometric interpretation) are indexed by $\bb \in \ZZ_+^k$ such
that $\sum_{i=1}^{k} b_i = q-1$: we have
	$$E_{k,q} = \left\{ e(\bb): \bb \in \ZZ^k, \bb \geq 0,
		 \sum_{i=1}^{k} b_i = q-1 \right\}$$
where
	$ e(\bb) = \{ \bb + \be_1, \bb + \be_2, \ldots, \bb + \be_k \}
		=	\{ (b_1+1, b_2, \ldots, b_k), (b_1, b_2+1,\ldots,b_k), \ldots, (b_1,b_2,\ldots, b_k+1) \}.$
%We sometimes omit the indices $k,q$ when there is no danger of confusion.
For each vertex $\ba \in V_{k,q}$, we have a list of admissible colors $L(\ba)$, which
is
	$$ L(\ba) = \{ i \in [k]: a_i > 0 \}.$$
%Formally, in the setting of the \HyperLabel problem, we define the
%assignment cost to be $c(\ba,i) = 0$ whenever $i \in L(\ba)$ and
%$c(\ba,i) = \infty$ otherwise.

\section{Sperner's Lemma}
\label{sec:Sperner}

First, let us recall the statement of Sperner's Lemma \cite{Sperner28}. 
We consider labelings $\ell:V_{k,q} \rightarrow [k]$. We call a labeling $\ell$ Sperner-admissible
if $\ell(\ba) \in L(\ba)$ for each $\ba \in V$; i.e.~, if $\ell(\ba) = j$ then $a_j > 0$.

\begin{lemma}[Sperner's Lemma]
For every Sperner-admissible labeling of the vertices of a simplicial subdivision of $\Delta_k$,
there is a cell whose vertices receive all $k$ colors.
\end{lemma}

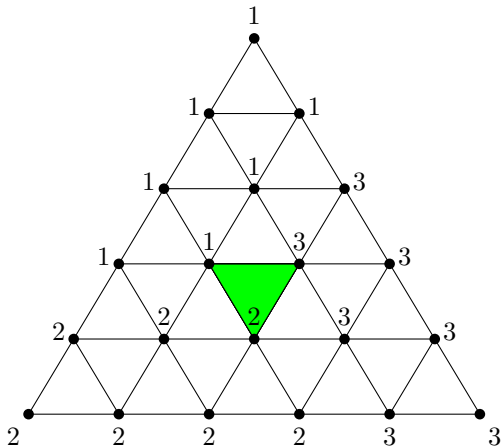
\begin{figure}[h]
\caption{\small A Sperner-admissible labeling for $k=3$ and $q=5$.
 %The set $E$ of hyperedges consists of the shaded triangles.
% The gray triangles are non-monochromatic hyperedges.
At least one cell in the triangulation (not necessarily in $E_{k,q}$) must be $k$-colored (rainbow).
}

\label{fig:Sperner}

\begin{tikzpicture}

\draw (-1,0) node {   };

%\draw [color=gray] (5,6.2) -- (0.5,-0.5) -- (9.5,-0.5) -- (5,6.2);
%\draw (1.0,-0.2) node {$V$};

\draw (2,0) -- (3.2,0) -- (2.6,1) -- cycle;
\draw (3.2,0) -- (4.4,0) -- (3.8,1) -- cycle;
\draw (4.4,0) -- (5.6,0) -- (5,1) -- cycle;
\draw (5.6,0) -- (6.8,0) -- (6.2,1) -- cycle;
\draw (6.8,0) -- (8,0) -- (7.4,1) -- cycle;

\draw (2.6,1) -- (3.8,1) -- (3.2,2) -- cycle;
\draw (3.8,1) -- (5,1) -- (4.4,2) -- cycle;
\draw (5,1) -- (6.2,1) -- (5.6,2) -- cycle;
\draw (6.2,1) -- (7.4,1) -- (6.8,2) -- cycle;

%rainbow triangle
\filldraw[fill=red] (4.4,2) -- (5.6,2) -- (5,1) -- cycle;
\filldraw[fill=blue,path fading=west] (4.4,2) -- (5.6,2) -- (5,1) -- cycle;
\filldraw[fill=green,path fading=north] (4.4,2) -- (5.6,2) -- (5,1) -- cycle;

\draw (3.2,2) -- (4.4,2) -- (3.8,3) -- cycle;
\draw (4.4,2) -- (5.6,2) -- (5,3) -- cycle;
\draw (5.6,2) -- (6.8,2) -- (6.2,3) -- cycle;
\draw (3.8,3) -- (5,3) -- (4.4,4) -- cycle;
\draw (5,3) -- (6.2,3) -- (5.6,4) -- cycle;
\draw (4.4,4) -- (5.6,4) -- (5,5) -- cycle;

\fill (2,0) circle (2pt);
\fill (3.2,0) circle (2pt);
\fill (4.4,0) circle (2pt);
\fill (5.6,0) circle (2pt);
\fill (6.8,0) circle (2pt);
\fill (8,0) circle (2pt);

\fill (2.6,1) circle (2pt);
\fill (3.8,1) circle (2pt);
\fill (5,1) circle (2pt);
\fill (6.2,1) circle (2pt);
\fill (7.4,1) circle (2pt);

\fill (3.2,2) circle (2pt);
\fill (4.4,2) circle (2pt);
\fill (5.6,2) circle (2pt);
\fill (6.8,2) circle (2pt);

\fill (3.8,3) circle (2pt);
\fill (5,3) circle (2pt);
\fill (6.2,3) circle (2pt);

\fill (4.4,4) circle (2pt);
\fill (5.6,4) circle (2pt);

\fill (5,5) circle (2pt);

\draw (5,5.3) node {$1$};
\draw (1.8,-0.3) node {$2$};
\draw (8.2,-0.3) node {$3$};

\draw (4.2,4.1) node {$1$};
\draw (3.6,3.1) node {$1$};
\draw (3,2.1) node {$1$};
\draw (2.4,1.1) node {$2$};

\draw (5.8,4.1) node {$1$};
\draw (6.4,3.1) node {$3$};
\draw (7,2.1) node {$3$};
\draw (7.6,1.1) node {$3$};

\draw (3.2,-0.3) node {$2$};
\draw (4.4,-0.3) node {$2$};
\draw (5.6,-0.3) node {$2$};
\draw (6.8,-0.3) node {$3$};

\draw (5,3.3) node {$1$};
\draw (4.4,2.3) node {$1$};
\draw (5.6,2.3) node {$3$};
\draw (3.8,1.3) node {$2$};
\draw (5,1.3) node {$2$};
\draw (6.2,1.3) node {$3$};

\end{tikzpicture}

\end{figure}

We remark that this does not say anything about the Simplex-Lattice Hypergraph:
Even if the subdivision uses the point set $V_{k,q}$, the rainbow cell given by Sperner's Lemma
might not be a member of $E_{k,q}$ since $E_{k,q}$ consists only of scaled copies
of $\Delta_{k,q}$ without rotation; it is not a full subdivision of the simplex.
(See Figure~\ref{fig:Sperner}.)

\section{The Simplex-Lattice Coloring Lemma}
\label{sec:lattice-coloring}

Instead of rainbow cells, the statement proposed (and proved for $k=3$)
 in \cite{EV14} involves non-monochromatic cells.

\begin{proposition}[Simplex-Lattice Coloring Lemma]
\label{prop:EV}
	For any Sperner-admissible labeling $\ell:V_{k,q} \rightarrow [k]$,
	there are at least ${q+k-3 \choose k-2}$ hyperedges $e \in E_{k,q}$
	that are non-monochromatic under $\ell$.
\end{proposition}

\paragraph{The first-choice labeling.}
In particular, Proposition~\ref{prop:EV} is that a Sperner-admissible labeling
minimizing the number of non-monochromatic cells is a ``first-choice one" which
labels each vertex $\ba$ by the smallest coordinate $i$ such that $a_i > 0$.
Under this labeling, all the hyperedges $e(\bb)$ such
that $b_1 > 0$ are labeled monochromatically by $1$. The only
hyperedges that receive more than $1$ color are those where $b_1 =
0$, and the number of such hyperedges is exactly ${q+k-3 \choose k-2}$
(see \cite{EV14}). Here we give a proof of Proposition~\ref{prop:EV}.

\begin{figure}[h]

\caption{\small The first-choice labeling.}

\begin{tikzpicture}

\draw (-1,0) node {   };

%\draw [color=gray] (5,6.2) -- (0.5,-0.5) -- (9.5,-0.5) -- (5,6.2);
%\draw (1.0,-0.2) node {$V$};

\filldraw [fill=gray] (2,0) -- (3.2,0) -- (2.6,1) -- cycle;
\filldraw [fill=gray] (3.2,0) -- (4.4,0) -- (3.8,1) -- cycle;
\filldraw [fill=gray] (4.4,0) -- (5.6,0) -- (5,1) -- cycle;
\filldraw [fill=gray] (5.6,0) -- (6.8,0) -- (6.2,1) -- cycle;
\filldraw [fill=gray] (6.8,0) -- (8,0) -- (7.4,1) -- cycle;

\filldraw [fill=green] (2.6,1) -- (3.8,1) -- (3.2,2) -- cycle;
\filldraw [fill=green] (3.8,1) -- (5,1) -- (4.4,2) -- cycle;
\filldraw [fill=green] (5,1) -- (6.2,1) -- (5.6,2) -- cycle;
\filldraw [fill=green] (6.2,1) -- (7.4,1) -- (6.8,2) -- cycle;

\filldraw [fill=green] (3.2,2) -- (4.4,2) -- (3.8,3) -- cycle;
\filldraw [fill=green] (4.4,2) -- (5.6,2) -- (5,3) -- cycle;
\filldraw [fill=green] (5.6,2) -- (6.8,2) -- (6.2,3) -- cycle;

\filldraw [fill=green] (3.8,3) -- (5,3) -- (4.4,4) -- cycle;
\filldraw [fill=green] (5,3) -- (6.2,3) -- (5.6,4) -- cycle;

\filldraw [fill=green] (4.4,4) -- (5.6,4) -- (5,5) -- cycle;

\fill (2,0) circle (2pt);
\fill (3.2,0) circle (2pt);
\fill (4.4,0) circle (2pt);
\fill (5.6,0) circle (2pt);
\fill (6.8,0) circle (2pt);
\fill (8,0) circle (2pt);

\fill (2.6,1) circle (2pt);
\fill (3.8,1) circle (2pt);
\fill (5,1) circle (2pt);
\fill (6.2,1) circle (2pt);
\fill (7.4,1) circle (2pt);

\fill (3.2,2) circle (2pt);
\fill (4.4,2) circle (2pt);
\fill (5.6,2) circle (2pt);
\fill (6.8,2) circle (2pt);

\fill (3.8,3) circle (2pt);
\fill (5,3) circle (2pt);
\fill (6.2,3) circle (2pt);

\fill (4.4,4) circle (2pt);
\fill (5.6,4) circle (2pt);

\fill (5,5) circle (2pt);

\draw (5,5.3) node {$1$};
\draw (1.8,-0.3) node {$2$};
\draw (8.2,-0.3) node {$3$};

\draw (4.2,4.1) node {$1$};
\draw (3.6,3.1) node {$1$};
\draw (3,2.1) node {$1$};
\draw (2.4,1.1) node {$1$};

\draw (5.8,4.1) node {$1$};
\draw (6.4,3.1) node {$1$};
\draw (7,2.1) node {$1$};
\draw (7.6,1.1) node {$1$};

\draw (3.2,-0.3) node {$2$};
\draw (4.4,-0.3) node {$2$};
\draw (5.6,-0.3) node {$2$};
\draw (6.8,-0.3) node {$2$};

\draw (5,3.3) node {$1$};
\draw (4.4,2.3) node {$1$};
\draw (5.6,2.3) node {$1$};
\draw (3.8,1.3) node {$1$};
\draw (5,1.3) node {$1$};
\draw (6.2,1.3) node {$1$};

\end{tikzpicture}

\end{figure}
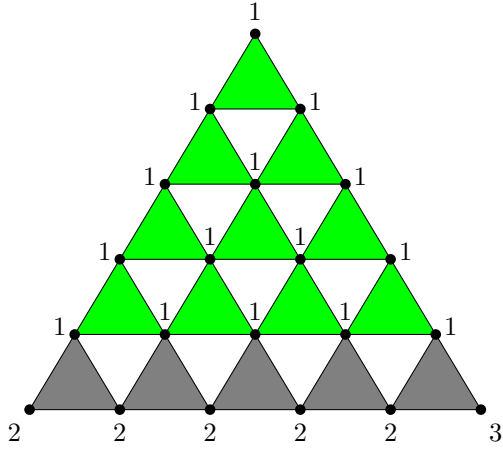

%\begin{lemma} \label{lem:EV-conjecture}
%	For any Sperner-admissible labeling $\ell:V_{k,q} \rightarrow [k]$,
%	there are at least ${q+k-3 \choose k-2}$ hyperedges $e \in E_{k,q}$
% that are non-monochromatic under $\ell$.
%\end{lemma}

\begin{proof}
Consider the set of hyperedges $E_{k,q}$: observe that it can be written naturally
as $$E_{k,q} = \{ e(\bb): \bb \in V_{k,q-1} \}.$$ I.e., the hyperedges can be identified one-to-one
with the vertices in $V_{k,q-1}$. Recall that $e(\bb) = \{ \bb + \be_1, \bb + \be_2, \ldots, \bb + \be_k \}$.
Two hyperedges $e(\bb), e(\bb')$ share a vertex
if and only if $\bb' + \be_j = \bb + \be_i$ for some pair $i,j \in [k]$; or in other words if $\bb, \bb'$
are nearest neighbors in $V_{k,q-1}$ (differ by $\pm 1$ in exactly two coordinates). 

Consider a labeling $\ell: V_{k,q} \rightarrow [k]$. For each $i \in [k]$, let $C_i$ denote the set of points in $V_{k,q-1}$
representing the monochromatic hyperedges in color $i$,
$$ C_i = \{ \bb \in V_{k,q-1}: \forall \bv \in e(\bb); \ell(\bv) = i \}. $$
Define an injective mapping $\phi_i:C_i \rightarrow V_{k,q-2}$ as follows:
$$ \phi_i(\bb) = \bb - \be_i.$$
The image is indeed in $V_{k,q-2}$: if $\bb \in C_i$, we have $b_i > 0$, or else $e(\bb)$
would contain a vertex $\ba$ such that $a_i = 0$ and hence $e(\bb)$ could not be monochromatic in color $i$.
Therefore, $\bb - \be_i \in \ZZ_+^k$ and $(\bb - \be_i) \cdot \b1 = q-2$ which means $\bb - \be_i \in V_{k,q-2}$.
(Here, $\b1$ denotes the all-$1$'s vector.)

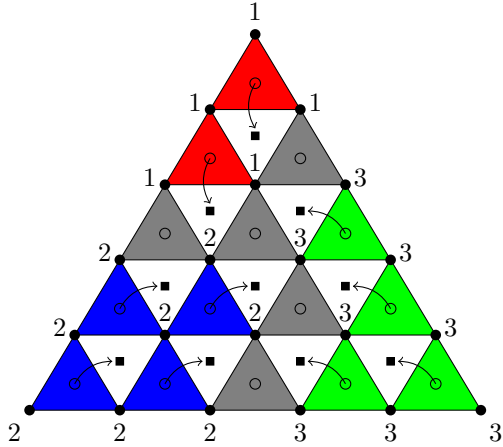
\begin{figure}[h]
\caption{\small The mappings $\phi_i:C_i \rightarrow V_{k,q-2}$. The hyperedges are represented by the empty circles; $C_i$ is the subset of them monochromatic in color $i$. The black squares represent $V_{k,q-2}$; note that each point in $V_{k,q-2}$ is the image of at most one monochromatic hyperedge.}

\begin{tikzpicture}

\draw (-1,0) node {   };

%\draw [color=gray] (5,6.2) -- (0.5,-0.5) -- (9.5,-0.5) -- (5,6.2);
%\draw (1.0,-0.2) node {$V$};

%hyperedges
\filldraw [fill=blue] (2,0) -- (3.2,0) -- (2.6,1) -- cycle;
\filldraw [fill=blue] (3.2,0) -- (4.4,0) -- (3.8,1) -- cycle;
\filldraw [fill=gray] (4.4,0) -- (5.6,0) -- (5,1) -- cycle;
\filldraw [fill=green] (5.6,0) -- (6.8,0) -- (6.2,1) -- cycle;
\filldraw [fill=green] (6.8,0) -- (8,0) -- (7.4,1) -- cycle;

\filldraw [fill=blue] (2.6,1) -- (3.8,1) -- (3.2,2) -- cycle;
\filldraw [fill=blue] (3.8,1) -- (5,1) -- (4.4,2) -- cycle;
\filldraw [fill=gray] (5,1) -- (6.2,1) -- (5.6,2) -- cycle;
\filldraw [fill=green] (6.2,1) -- (7.4,1) -- (6.8,2) -- cycle;

\filldraw [fill=gray] (3.2,2) -- (4.4,2) -- (3.8,3) -- cycle;
\filldraw [fill=gray] (4.4,2) -- (5.6,2) -- (5,3) -- cycle;
\filldraw [fill=green] (5.6,2) -- (6.8,2) -- (6.2,3) -- cycle;

\filldraw [fill=red] (3.8,3) -- (5,3) -- (4.4,4) -- cycle;
\filldraw [fill=gray] (5,3) -- (6.2,3) -- (5.6,4) -- cycle;

\filldraw [fill=red] (4.4,4) -- (5.6,4) -- (5,5) -- cycle;

%vertices
\fill (2,0) circle (2pt);
\fill (3.2,0) circle (2pt);
\fill (4.4,0) circle (2pt);
\fill (5.6,0) circle (2pt);
\fill (6.8,0) circle (2pt);
\fill (8,0) circle (2pt);

\fill (2.6,1) circle (2pt);
\fill (3.8,1) circle (2pt);
\fill (5,1) circle (2pt);
\fill (6.2,1) circle (2pt);
\fill (7.4,1) circle (2pt);

\fill (3.2,2) circle (2pt);
\fill (4.4,2) circle (2pt);
\fill (5.6,2) circle (2pt);
\fill (6.8,2) circle (2pt);

\fill (3.8,3) circle (2pt);
\fill (5,3) circle (2pt);
\fill (6.2,3) circle (2pt);

\fill (4.4,4) circle (2pt);
\fill (5.6,4) circle (2pt);

\fill (5,5) circle (2pt);

%labeling
\draw (5,5.3) node {$1$};
\draw (1.8,-0.3) node {$2$};
\draw (8.2,-0.3) node {$3$};

\draw (4.2,4.1) node {$1$};
\draw (3.6,3.1) node {$1$};
\draw (3,2.1) node {$2$};
\draw (2.4,1.1) node {$2$};

\draw (5.8,4.1) node {$1$};
\draw (6.4,3.1) node {$3$};
\draw (7,2.1) node {$3$};
\draw (7.6,1.1) node {$3$};

\draw (3.2,-0.3) node {$2$};
\draw (4.4,-0.3) node {$2$};
\draw (5.6,-0.3) node {$3$};
\draw (6.8,-0.3) node {$3$};

\draw (5,3.3) node {$1$};
\draw (4.4,2.3) node {$2$};
\draw (5.6,2.3) node {$3$};
\draw (3.8,1.3) node {$2$};
\draw (5,1.3) node {$2$};
\draw (6.2,1.3) node {$3$};

%hyperedge centers
\draw (2.6,0.35) circle (2pt);
\draw (3.8,0.35) circle (2pt);
\draw (5.0,0.35) circle (2pt);
\draw (6.2,0.35) circle (2pt);
\draw (7.4,0.35) circle (2pt);

\draw (3.2,1.35) circle (2pt);
\draw (4.4,1.35) circle (2pt);
\draw (5.6,1.35) circle (2pt);
\draw (6.8,1.35) circle (2pt);

\draw (3.8,2.35) circle (2pt);
\draw (5.0,2.35) circle (2pt);
\draw (6.2,2.35) circle (2pt);

\draw (4.4,3.35) circle (2pt);
\draw (5.6,3.35) circle (2pt);

\draw (5.0,4.35) circle (2pt);

%point of V_{k,q-2}
\draw [fill=black] (3.15,0.6) rectangle (3.25,0.7);
\draw [fill=black] (4.35,0.6) rectangle (4.45,0.7);
\draw [fill=black] (5.55,0.6) rectangle (5.65,0.7);
\draw [fill=black] (6.75,0.6) rectangle (6.85,0.7);

\draw [fill=black] (3.75,1.6) rectangle (3.85,1.7);
\draw [fill=black] (4.95,1.6) rectangle (5.05,1.7);
\draw [fill=black] (6.15,1.6) rectangle (6.25,1.7);

\draw [fill=black] (4.35,2.6) rectangle (4.45,2.7);
\draw [fill=black] (5.55,2.6) rectangle (5.65,2.7);

\draw [fill=black] (4.95,3.6) rectangle (5.05,3.7);

\path[->] (2.6,0.35) edge [bend left] (3.1,0.65);
\path[->] (3.8,0.35) edge [bend left] (4.3,0.65);
\path[->] (3.2,1.35) edge [bend left] (3.7,1.65);
\path[->] (4.4,1.35) edge [bend left] (4.9,1.65);

\path[->] (6.2,0.35) edge [bend right] (5.7,0.65);
\path[->] (7.4,0.35) edge [bend right] (6.9,0.65);
\path[->] (6.8,1.35) edge [bend right] (6.3,1.65);
\path[->] (6.2,2.35) edge [bend right] (5.7,2.65);

\path[->] (5.0,4.35) edge [bend right] (5.0,3.75);
\path[->] (4.4,3.35) edge [bend right] (4.4,2.75);

\end{tikzpicture}

\end{figure}

Further, we claim that $\phi_i[C_i] \cap \phi_j[C_j] = \emptyset$ for every $i \neq j$. If not, there would be $\bb \in C_i$
and $\bb' \in C_j$ such that $\bb - \be_i = \bb' - \be_j$.
 %i.e.~$\ba,\bb$ are nearest neighbors in $V_{k,q-2}$ and $e(\ba) \cap e(\bb) \neq \emptyset$.
Then, the point $\ba = \bb + \be_j = \bb' + \be_i$ would be an element of both the hyperedge $e(\bb)$ and the hyperedge $e(\bb')$.
This contradicts the assumption that $e(\bb)$ is monochromatic in color $i$
and $e(\bb')$ is monochromatic in color $j$. So the sets $\phi_i[C_i]$ are pairwise disjoint subsets of $V_{k,q-2}$.
By the definition of $\phi_i$, we clearly have $|\phi_i[C_i]| = |C_i|$.
We conclude that the total number of monochromatic hyperedges is
$$ \sum_{i=1}^{k} |C_i| = \sum_{i=1}^{k} |\phi_i[C_i]|  \leq |V_{k,q-2}|.$$
The total number of hyperedges is $|E_{k,q}| = |V_{k,q-1}|$. Considering that $|V_{k,q}| = {q+k-1 \choose k-1}$ (the number of partitions of $q$ into a sum of $k$ nonnegative integers), we obtain
that the number of non-monochromatic hyperedges is
\begin{eqnarray*}
& |E_{k,q}| - \sum_{i=1}^{k} |C_i| \geq |V_{k,q-1}| - |V_{k,q-2}|  =   {q+k-2 \choose k-1} - {q+k-3 \choose k-1} = {q+k-3 \choose k-2}. &
\end{eqnarray*}
\end{proof}

\section{A labeling of $H_{k,q}$ with at most 4 colors on each hyperedge}
\label{sec:at-most-4}

We recall that Sperner's lemma states that any Sperner-admissible labeling of a subdivision of the simplex must contain a simplex with all $k$ colors. The hypergraph $H_{k,q}$ defined in Section~\ref{sec:sperner-simplex} is not a subdivision since it covers only a subset of the large simplex. It is easy to see that the conclusion of Sperner's lemma does not hold for $H_{k,q}$ --- for example for $k=3$, we can label a $2$-dimensional triangulation so that exactly one triangle has $3$ different colors, and this triangle is not in $E_{3,q}$.
(See Figure~2.)
Hence, each triangle in $E_{3,q}$ has at most $2$ colors. By an extension of this argument, we can label $H_{k,q}$ so that each hyperedge in $E_{k,q}$ contains at most $k-1$ colors. The question we ask in this section is, what is the minimum $c$ such that there is a Sperner-admissible labeling with at most $c$ different colors on each hyperedge in $E_{k,q}$? We prove the following result.

\begin{proposition}
\label{prop:4-colors}
For any $k \geq 4$ and $q \geq k^2$, there is a Sperner-admissible labeling of $H_{k,q} = (V_{k,q}, E_{k,q})$ such that every hyperedge in $E_{k,q}$ contains at most $4$ different colors.
\end{proposition}

We note that this statement is not true for $q=1$ and $k > 4$ (since $E_{k,1}$ consists of a single simplex which has $k$ different colors). We have not identified the optimal lower bound on $q$ that allows our statement to hold.
Also, the statement could possibly hold with $2$ or $3$ colors instead of $4$; the number $4$ is just an artifact of our proof and we have no reason to believe that it is tight.
%The most basic remaining question is whether the number of colors in each hyperedge can be bounded by $2$ for every $k$ and sufficiently large $q$.

The intuition behind our construction is as follows:
We want to label the vertices so that the number of different colors on each hyperedge is small.
A natural choice is to label each vertex $\bv$ by its maximum-value coordinate. However, this does not work since
a hyperedge in the center of the simplex may receive all $k$ colors. The problem is that this labeling is possibly very sensitive to small changes in $\bv$. A more ``robust" labeling is one where we select a subset of ``top coordinates" and choose one among them according to another rule. This rule should be such that incrementing the coordinates one at a time does not change the label too many times. One such rule that works well is described below.

\begin{proof}
%Let $k \geq 4$ and $q \geq k^2$. 
We define a labeling $\ell:V_{k,q} \rightarrow [k]$ as follows:
\begin{itemize}
\item Given $\ba \in V_{k,q}$, let $\pi:[k] \rightarrow [k]$ be a permutation such that $a_{\pi(1)} \geq a_{\pi(2)} \geq \ldots \geq a_{\pi(k)}$
(and if $a_{\pi(i)} = a_{\pi(i+1)}$, we order $\pi$ so that $\pi(i) < \pi(i+1)$).
\item Define $t(\ba)$ to be the maximum $t \in [k]$ such that $\forall 1 \leq j \leq t$, $a_{\pi(j)} \geq k-j+1$. We define the ``Top coordinates" of $\ba$ to be $Top(\ba) = (\pi(1), \ldots, \pi(t(\ba)))$ (an ordered set).
\item We define the label of $\ba$ to be $\ell(\ba) = \pi(t(\ba))$, the index of the ``last Top coordinate".
\end{itemize}
First, we verify that this is a well-defined Sperner-admissible labeling. Since $\sum_{i=1}^{k} a_i = q \geq k^2$, we have $a_{\pi(1)} = \max a_i \geq k$ and hence $1 \leq t(\ba) \leq k$. For each $\ba \in V_{k,q}$, we have: $a_{\ell(\ba)} = a_{\pi(t(\ba))} \geq k-t(\ba)+1 > 0$, since $t(\ba) \leq k$. Therefore, $\ell$ is Sperner-admissible. 

Now, consider a hyperedge $e(\bb) = (\bb + \be_1, \bb + \be_2, \ldots, \bb + \be_k)$ where $\bb \geq 0, \sum_{i=1}^{k} b_i = q-1$.
We claim that $\ell(\bb + \be_i)$ attains at most $4$ different values for $i=1,\ldots,k$.
Without loss of generality, assume that $b_1 \geq b_2 \geq \ldots \geq b_k$.
Define $\ell^*$ to be the label assigned to $\bb$ by our construction (note that $\bb$ is not a vertex in $V_{k,q}$ but we can still apply our definition):
$\ell^*$ is the maximum value in $[k]$ such that for all $1 \leq j \leq \ell^*$, $b_j \geq k-j+1$.
Hence, we have $Top(\bb) = \{1,2,\ldots,\ell^*\}$.

Let $i \in [k]$, $\ba = \bb + \be_i$, and let $\pi$ be the permutation such that $a_{\pi(1)} \geq \ldots \geq a_{\pi(k)}$ as above.
(Recall that for $\bb$, we assumed that the respective permutation is the identity.)
We consider the following cases:
\begin{itemize}
\item If $1 \leq i < \ell^*$, then we claim that $\ell(\ba) = \ell(\bb+\be_i) = \ell(\bb) = \ell^*$. In the rule for selecting $t(\ba)$, one of the first $\ell^*-1$ coordinates has been incremented compared to $\bb$, which possibly pushes $i$ forward in the ordering of the Top coordinates. However, the other coordinates remain unchanged, the condition $a_{\pi(j)} \geq k-j+1$ is still satisfied for $1 \leq j \leq \ell^*$, and $Top(\ba) = Top(\bb)$. In particular $\ell^*$ is still the last coordinate included in $Top(\ba)$ and hence $\ell(\ba) = \ell^*$.

\item If $i = \ell^*$, then $\ell(\ba) = \ell(\bb+\be_{\ell^*})$ is still one of the coordinates in $Top(\bb)$, possibly different from $\ell^*$ (due to a change in order, although we still have $Top(\ba) = Top(\bb)$) --- let us call this label $\ell^*_2$.

\item If $\ell^* < i \leq k$, then it is possible that in $\ba = \bb + \be_i$, we obtain additional Top coordinates ($Top(\ba) \supset Top(\bb)$). It could be $a_i = b_i + 1$ itself which is now included among the Top coordinates, and possibly additional coordinates that already satisfied the condition $b_j \geq k-j+1$ but were not selected due to the condition being false for $b_{\ell^*+1}$. If this does not happen and we have $Top(\ba) = Top(\bb)$, the label of $\ba$ is still $\ell(\ba) = \ell^*$ (because the ordering of the Top coordinates remains the same).

Assume now that $Top(\ba)$ has additional coordinates beyond $Top(\bb)$.
By the definition of $\ell^*$, we have $b_{\ell^*} \geq k-\ell^*+1$, and for each $j > \ell^*$, we have $b_j < k-\ell^*$; otherwise $j$ would have been still chosen in $Top(\bb)$. For $Top(\ba) = Top(\bb+\be_i)$ to grow beyond $Top(\bb)$, $a_i$ must become the $(\ell^*+1)$-largest coordinate and satisfy $a_i \geq k - \ell^*$.
The only way this can happen is that $b_i = k-\ell^*-1$ and hence $a_i = b_i + 1 = k-\ell^*$. In this case, $a_i$ is the maximum coordinate among $\{a_j: j > \ell^*\}$, and still smaller than $a_{\ell^*}$.
Therefore, $i$ will be included in $Top(\ba)$. Now, $Top(\ba)$ may grow further. However, note that the construction of $Top(\ba)$  will proceed in the same way for every $\ba = \bb + \be_i$ such that $b_i = k-\ell^*-1$. This is because all the coordinates equal to $k-\ell^*-1$ will be certainly included in $Top(\ba)$, and coordinates smaller than $k - \ell^* - 1$ remain the same in each of these cases (equal to the coordinates of $\bb$).
Therefore, the set $Top(\ba)$ will be the same in all these cases; let us call this set $Top_+$.

The label assigned to $\ba = \bb + \be_i$ is the index of the last coordinate included in $Top_+ = Top(\ba)$. Since $Top_+$ is the same whenever $Top(\ba) \neq Top(\bb)$, the label of $\ba$ will be the coordinate $j^*$ minimizing $b_j$ (and maximizing $j$ to break ties) among all $j \in Top_+$, unless $j^* = i$ in which case the last included coordinate might be another one. This gives potentially two additional colors, let us call them $\ell^*_3, \ell^*_4$, that are assigned to $\ba = \bb + \be_i$ for all $i > \ell^*$ where $b_i = k - \ell^* - 1$. For other choices of $i > \ell^*$, we have $Top(\bb+\be_i) = Top(\bb)$ and the label assigned to $\bb + \be_i$ is $\ell(\bb+\be_i) = \ell^*$.
\end{itemize}
To summarize, all the colors that appear in the labeling of $e(\bb)$ are included in $\{\ell^*, \ell^*_2, \ell^*_3, \ell^*_4 \}$.
\end{proof}

\section{Boundary-minimizing partitioning of the simplex}
\label{sec:simplex-partition}

Let us turn now to a geometric variant of Proposition~\ref{prop:EV}. We recall that Sperner's Lemma has a geometric variant known as the Knaster-Kuratowski-Mazurkiewicz Lemma \cite{KKM29}:

\medskip

\noindent
{\em Consider a covering of the simplex $\Delta_k$ by closed sets $A_1,\ldots,A_k$ such that each point $\bx \in \Delta_k$ is contained in some set $A_i$ such that $x_i > 0$. Then $\bigcap_{i=1}^{k} A_i \neq \emptyset$.}

\medskip

Here we consider a similar setup, but instead of the intersection of all sets, we are interested in the measure of the boundaries between pairs of adjacent sets. 
%Instead of coloring the lattice $V_{k,q}$, let us consider colorings of the full simplex $\Delta_k$, subject to the same boundary conditions: For $\bx \in \Delta_k$, we are allowed to use only the colors $i \in [k]$ such that $x_i > 0$. Equivalently, we can talk about partitioning $\Delta_k$ into $k$ parts $A_1,\ldots,A_k$ such that $A_i$ contains the vertex $\be_i$ and $A_i$ is disjoint from the opposite facet defined by $x_i = 0$. 
To avoid technicalities, let us assume that the $A_i$'s are closed, disjoint except on the boundary,
and each $A_i$ is disjoint from the face $\{ \bx \in \Delta_k: x_i = 0\}$.

\begin{definition}
\label{def:Sperner-partition}
A Sperner-admissible partition of $\Delta_k$ is a $k$-tuple of closed sets $(A_1,\ldots,A_k)$ such that
\begin{compactitem}
\item $\bigcup_{i=1}^{k} A_i = \Delta_k$,
\item $A_1, \ldots, A_k$ are disjoint except on their boundary,
\item $x_i > 0$ for every $\bx \in A_i$.
\end{compactitem}
We call the union of pairwise boundaries $\bigcup_{i \neq j} (A_i \cap A_j)$ the {\em \sep}.
\end{definition}

The question we ask here is, in analogy with Proposition~\ref{prop:EV}, what is the Sperner-admissible partition with the \sep~of minimum measure? 
%By a \sep, we mean the union of all the pairwise boundaries, $\bigcup_{i \neq j} (A_i \cap A_j)$.
A candidate partition is depicted in Figure~\ref{fig:Voronoi}, where $A_i$ is the set of all points in $\Delta_k$ for whom $\be_i$ is the closest vertex.
We call this the {\em Voronoi partition}.

\begin{figure}[h]
\caption{The Voronoi partition of a simplex.}
\label{fig:Voronoi}
\begin{tikzpicture}

\draw (-3,0) node {   };

%\draw [color=gray] (5,6.2) -- (0.5,-0.5) -- (9.5,-0.5) -- (5,6.2);
%\draw (1.0,-0.2) node {$V$};

\def\N{50}
\def\Ax{0}
\def\Ay{0}
\def\Bx{6}
\def\By{0}
\def\Cx{6.5}
\def\Cy{1.5}
\def\Dx{3}
\def\Dy{5}

\coordinate (A) at (\Ax, \Ay);
\coordinate (B) at (\Bx, \By);
\coordinate (C) at (\Cx, \Cy);
\coordinate (D) at (\Dx, \Dy);

\coordinate(AB) at (0.5*\Ax+0.5*\Bx, 0.5*\Ay+0.5*\By);
\coordinate(BC) at (0.5*\Bx+0.5*\Cx, 0.5*\By+0.5*\Cy);
\coordinate(CD) at (0.5*\Cx+0.5*\Dx, 0.5*\Cy+0.5*\Dy);
\coordinate(AD) at (0.5*\Ax+0.5*\Dx, 0.5*\Ay+0.5*\Dy);
\coordinate(AC) at (0.5*\Ax+0.5*\Cx, 0.5*\Ay+0.5*\Cy);
\coordinate(BD) at (0.5*\Bx+0.5*\Dx, 0.5*\By+0.5*\Dy);

\coordinate(ABC)  at (0.333*\Ax+0.333*\Bx+0.333*\Cx, 0.333*\Ay+0.333*\By+0.333*\Cy);
\coordinate(ABD)  at (0.333*\Ax+0.333*\Bx+0.333*\Dx, 0.333*\Ay+0.333*\By+0.333*\Dy);
\coordinate(ACD)  at (0.333*\Ax+0.333*\Cx+0.333*\Dx, 0.333*\Ay+0.333*\Cy+0.333*\Dy);
\coordinate(BCD)  at (0.333*\Bx+0.333*\Cx+0.333*\Dx, 0.333*\By+0.333*\Cy+0.333*\Dy);

\coordinate(Z)  at (0.25*\Ax+0.25*\Bx+0.25*\Cx+0.25*\Dx, 0.25*\Ay+0.25*\By+0.25*\Cy+0.25*\Dy);

\draw [gray] (Z) -- (AB);
\draw [gray] (Z) -- (AC);
\draw [gray] (Z) -- (AD);
\draw [gray] (Z) -- (BC);
\draw [gray] (Z) -- (BD);
\draw [gray] (Z) -- (CD);

\filldraw [cyan, opacity=0.2] (A) -- (C) -- (D) -- cycle;
\filldraw [cyan, opacity=0.2] (A) -- (B) -- (C) -- cycle;
\filldraw [cyan, opacity=0.2] (Z) -- (ACD) -- (CD) -- (BCD) -- cycle;
\filldraw [cyan, opacity=0.2] (Z) -- (ABC) -- (BC) -- (BCD) -- cycle;
\filldraw [cyan, opacity=0.2] (Z) -- (ABD) -- (AD) -- (ACD) -- cycle;
\filldraw [cyan, opacity=0.2] (Z) -- (ABD) -- (BD) -- (BCD) -- cycle;
\filldraw [cyan, opacity=0.2] (Z) -- (ABC) -- (AC) -- (ACD) -- cycle;
\filldraw [cyan, opacity=0.2] (Z) -- (ABC) -- (AB) -- (ABD) -- cycle;

\draw (ABC) -- (AB) -- (ABD);
\draw (ABC) -- (AC) -- (ACD);
\draw (ABC) -- (BC) -- (BCD);
\draw (ABD) -- (AD) -- (ACD);
\draw (ABD) -- (BD) -- (BCD);
\draw (ACD) -- (CD) -- (BCD);
\draw (ABC) -- (Z) -- (ABD);
\draw (ACD) -- (Z) -- (BCD);

\color{black}
\draw [thick] (A) -- (B) -- (C) -- (D) -- (A);
\draw [thick] (A) -- (C);
\draw [thick] (B) -- (D);

\color{black}
%vertices
\fill (A) circle (2pt);
\fill (B) circle (2pt);
\fill (C) circle (2pt);
\fill (D) circle (2pt);

%labeling
\draw (A) [left] node {$\be_1$};
\draw (B) [right]  node {$\be_2$};
\draw (C) [right] node {$\be_3$};
\draw (D) [right] node {$\be_4$};

\end{tikzpicture}
\end{figure}
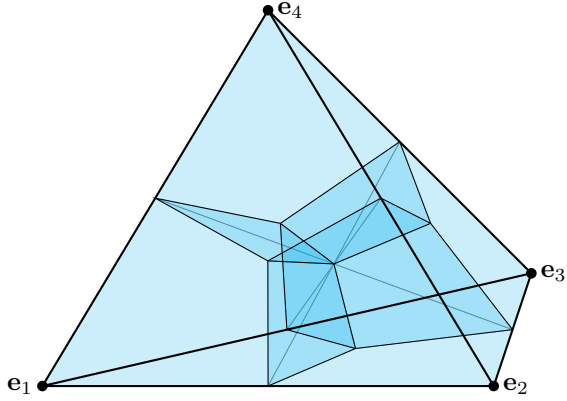

We prove that for the regular simplex $\Delta_k$ this is indeed the optimal partition (along with other, similar configurations).
In the following, we denote by $\mu_k$ the usual Lebesgue measure on $\RR^k$, and by $\mu_\ell$ ($\ell<k$) the $\ell$-dimensional Minkowski content.

\begin{definition}
For $A \subset \RR^k$, the $\ell$-dimensional Minkowski content is (if the limit exists)
$$ \mu_\ell(A) = \lim_{\epsilon \rightarrow 0^+} \frac{\mu_k(A_\epsilon)}{\alpha_{k-\ell} \epsilon^{k-\ell}} $$
where $A_\epsilon = \{ \by \in \RR^k: \exists \bx \in A, \|\bx-\by\| \leq \epsilon \}$ is the $\epsilon$-neighborhood of $A$
and $\alpha_{k-\ell}$ is the volume of a unit ball in $\RR^{k-\ell}$.
We also define $\mu_\ell^+(A)$ to be the upper limit and $\mu_\ell^-(A)$ the lower limit of the expression above.
\end{definition}

We remark that for $\ell$-rectifiable sets (polyhedral faces, smooth surfaces, etc.) the notion of Minkowski content coincides with that of Hausdorff measure (under suitable normalization).

\begin{theorem}
\label{thm:simplex-partition}
For every Sperner-admissible partition $(A_1,\ldots,A_k)$ of $\Delta_k$,
$$ \mu^-_{k-2}\left( \bigcup_{i \neq j} (A_i \cap A_j) \right) \geq \frac{k-1}{\sqrt{2}} \mu_{k-1}(\Delta_k) $$
and the Voronoi partition achieves this with equality.
% \geq \frac{k-1}{\sqrt{2}} \mu_{k-1}(\Delta_k)
% = \frac{1}{(k-2)!} \sqrt{\frac{k}{2}} $$
%and this bound is attained by the Voronoi partitining (Figure~\ref{fig:Voronoi}).
\end{theorem}

First, let us analyze the Voronoi partition and more generally the following kind of partition.

\begin{lemma}
\label{lem:Voronoi}
For any $\bz$ in the interior of $\Delta_k$, the partition $(A^z_1,\ldots,A^z_k)$ where
$$ A^z_i = \{ \bx \in \Delta_k: x_i-z_i = \max_{1 \leq j \leq k} (x_j - z_j) \} $$
satisfies
$$ \mu_{k-2}\left( \bigcup_{i \neq j} (A^z_i \cap A^z_j) \right)
= \frac{k-1}{\sqrt{2}} \mu_{k-1}(\Delta_k) = \frac{1}{(k-2)!} \sqrt{\frac{k}{2}}.$$
\end{lemma}

We call this kind of partition ``Voronoi-type".\footnote{We note that these partitions are also known
as ``power diagrams".}
We note that that for $\bz = (\frac{1}{k}, \frac{1}{k}, \ldots, \frac{1}{k})$ we obtain the Voronoi partition in Figure~\ref{fig:Voronoi}. Other choices of $\bz$ correspond to similar configurations where all the colors meet at the point $\bz$. Note that $\bz$ is the ``rainbow point" guaranteed by the Knaster-Kuratowski-Mazurkiewicz Lemma. 

\begin{proof}
First let us compute some basic quantities that we will need. The sides of our simplex $\Delta_k$ have length $\sqrt{2}$. Denote by $h_k$ the height of $\Delta_k$, that is the distance of any vertex from the opposite facet. We have
$$ h_k = \left\| (1,0,\ldots,0) - (0,\frac1{k-1},\ldots,\frac1{k-1}) \right\| = \sqrt{1 + (k-1) \cdot \frac{1}{(k-1)^2}} = \sqrt{\frac{k}{k-1}}.$$
The volume of the simplex can be computed inductively as follows; we have $\mu_1(\Delta_2) = \sqrt{2}$, and
$ \mu_k(\Delta_{k+1}) = \frac{1}{k} h_{k+1} \cdot \mu_{k-1}(\Delta_{k}).$
This implies $$\mu_{k-1}(\Delta_{k}) = \frac{\sqrt{k}}{(k-1)!}.$$
Now let us compute the measure of the \sep~for the partition $(A^z_1,\ldots,A^z_k)$ defined above, by induction. The \sep~can be described explicitly as
$$ \bigcup_{i \neq j} (A^z_i \cap A^z_j) = \{ \bx \in \Delta_k: \exists i \neq j, x_i-z_i = x_j-z_j = \max_{1 \leq \ell \leq k} x_\ell-z_\ell \}.$$
For $k=2$, $A^z_1 \cap A^z_2$ is just a single point, and $\mu_0(A^z_1 \cap A^z_2) = 1$.
For $k \geq 3$, denote by $S$ the \sep~for $(A^z_1,\ldots,A^z_k)$ and define $S_i = S \cap \mbox{conv}(\{\be_j: j \neq i\})$, the \sep~restricted to the facet opposite vertex $\be_i$. Since $S_i$ is a Voronoi-type \sep~for $\Delta_{k-1}$, by induction we assume that $\mu_{k-3}(S_i) = \frac{1}{(k-3)!} \sqrt{\frac{k-1}{2}}$. The \sep~$S$ can be written as $S = \bigcup_{i=1}^{k} \mbox{conv}(S_i \cup \{\bz\})$, see Figure~\ref{fig:Voronoi}. Denote by $h'_i$ the distance of $\bz$ from the facet containing $S_i$. By the pyramid formula in dimension $k-2$, 
$$\mu_{k-2}(\mbox{conv}(S_i \cup \{\bz\})) = \frac{1}{k-2} h'_i \mu_{k-3}(S_i) = \frac{h'_i}{(k-2)!} \sqrt{\frac{k-1}{2}}.$$
By a simple calculation, $\sum_{i=1}^{k} h'_i = h_k = \sqrt{\frac{k}{k-1}}$. 
The sets $\mbox{conv}(S_i \cup \{\bz\})$ are disjoint except for lower-dimensional intersections. Hence,
$$ \mu_{k-2}(S) = \sum_{i=1}^{k} \mu_{k-2}(\mbox{conv}(S_i \cup \{\bz\}))
 = \sum_{i=1}^{k} \frac{h'_i}{(k-2)!} \sqrt{\frac{k-1}{2}} = \frac{1}{(k-2)!} \sqrt{\frac{k}{2}}.$$
\end{proof}

Thus the proof of Theorem~\ref{thm:simplex-partition} will be complete if we prove the following bound.

\begin{lemma}
\label{lem:sep-bound}
For every Sperner-admissible partition $(A_1,\ldots,A_k)$ of $\Delta_k$,
$$ \mu_{k-2}^-\left(\bigcup_{i \neq j} (A_i \cap A_j)\right) \geq \frac{k-1}{\sqrt{2}} \mu_{k-1}(\Delta_k)
 = \frac{1}{(k-2)!} \sqrt{\frac{k}{2}}.$$
\end{lemma}

\begin{proof}
We pursue an approach similar to the proof of Proposition~\ref{prop:EV}, with some additional technicalities.
The high-level approach is to shrink the sets $A_i$ somewhat, by excluding a small neighborhood of the \sep. This creates a buffer zone between the shrunk sets $A'_i$ (yellow in Figure~\ref{fig:shrinking}) whose measure corresponds to the measure of the \sep. Since we have this extra space, we are able to push the sets $A'_i$ closer together and obtain sets $A''_i$ that fit inside a slightly smaller simplex. The difference between the volume of this simplex and the original one gives a bound on the measure of the \sep.

First, let $\epsilon_0 = \inf_{i \in [k], \bx \in A_i} x_i$. Recall that $x_i > 0$ for each $\bx \in A_i$, and moreover each $A_i$ is closed. 
%If $\epsilon_0 = 0$ then there is a sequence of points $\bx^{(n)} \in A_i$ such that $\lim_{n \rightarrow \infty} x^{(n)}_i = 0$. But then there is a point $\bx^* \in A_i$, the limit of a convergent subsequence, such that $x^*_i = 0$ which is a contradiction. 
Hence $\epsilon_0 > 0$.

Define $S = \bigcup_{i \neq j }(A_i \cap A_j)$, the \sep~whose measure we are trying to lower-bound.
Fix $\epsilon \in (0,\frac12\epsilon_0)$ (eventually we will let $\epsilon \rightarrow 0$)
and define $S_\epsilon$ as the $\epsilon$-neighborhood of $S$,
$$ S_\epsilon = \left\{ \bx \in \Delta_k: \exists \by \in S, \|\bx-\by\| \leq \epsilon \right\}.$$
We define subsets $A'_i \subset A_i$ as follows:
$$ A'_i = A_i \setminus S_\epsilon.$$
Thus we have $\bigcup_{i=1}^{k} A'_i = \Delta_k \setminus S_\epsilon$. Also, the sets $A'_i$ are clearly disjoint (see Figure~\ref{fig:shrinking}).

\begin{figure}[h]

\caption{The construction of $A'_i$ and $A''_i$.}

\label{fig:shrinking}

\begin{tikzpicture}

\def\Ax{0}
\def\Ay{0}
\def\Bx{5}
\def\By{0}
\def\Cx{2.5}
\def\Cy{4}

\coordinate (A) at (\Ax, \Ay);
\coordinate (B) at (\Bx, \By);
\coordinate (C) at (\Cx, \Cy);

\coordinate (AB) at (0.5*\Ax + 0.5*\Bx, 0.5*\Ay + 0.5*\By);
\coordinate (ABA) at (0.52*\Ax + 0.48*\Bx, 0.52*\Ay + 0.48*\By);
\coordinate (ABB) at (0.48*\Ax + 0.52*\Bx, 0.48*\Ay + 0.52*\By);

\coordinate (AC) at (0.5*\Ax + 0.5*\Cx, 0.5*\Ay + 0.5*\Cy);
\coordinate (ACA) at (0.52*\Ax + 0.48*\Cx, 0.52*\Ay + 0.48*\Cy);
\coordinate (ACC) at (0.48*\Ax + 0.52*\Cx, 0.48*\Ay + 0.52*\Cy);

\coordinate (BC) at (0.5*\Cx + 0.5*\Bx, 0.5*\Cy + 0.5*\By);
\coordinate (BCC) at (0.52*\Cx + 0.48*\Bx, 0.52*\Cy + 0.48*\By);
\coordinate (BCB) at (0.48*\Cx + 0.52*\Bx, 0.48*\Cy + 0.52*\By);

\coordinate (Z) at (0.333*\Ax + 0.333*\Bx + 0.333*\Cx, 0.333*\Ay + 0.333*\By + 0.333*\Cy);
\coordinate (ZA) at (0.36*\Ax + 0.32*\Bx + 0.32*\Cx, 0.36*\Ay + 0.32*\By + 0.32*\Cy);
\coordinate (ZB) at (0.32*\Ax + 0.36*\Bx + 0.32*\Cx, 0.32*\Ay + 0.36*\By + 0.32*\Cy);
\coordinate (ZC) at (0.32*\Ax + 0.32*\Bx + 0.36*\Cx, 0.32*\Ay + 0.32*\By + 0.36*\Cy);

\filldraw [yellow] (A) -- (B) -- (C) -- cycle;

\filldraw [cyan, opacity=1] (A) -- (ABA) -- (ZA) -- (ACA) -- cycle;
\filldraw [green, opacity=1] (B) -- (ABB) -- (ZB) -- (BCB) -- cycle;
\filldraw [red, opacity=1] (C) -- (ACC) -- (ZC) -- (BCC) -- cycle;

\draw (Z) -- (AB);
\draw (Z) -- (AC);
\draw (Z) -- (BC);

\draw (ZA) -- (ABA);
\draw (ZA) -- (ACA);
\draw (ZB) -- (ABB);
\draw (ZB) -- (BCB);
\draw (ZC) -- (ACC);
\draw (ZC) -- (BCC);

\color{black}
\draw [thick] (A) -- (B) -- (C) -- cycle;

\color{black}
%vertices
\fill (A) circle (2pt);
\fill (B) circle (2pt);
\fill (C) circle (2pt);

%labeling
\draw (A) [left] node {$\be_1$};
\draw (B) [right]  node {$\be_2$};
\draw (C) [right] node {$\be_3$};

%\draw (ZA) [left] node {$A_1$};
%\draw (ZB) [right] node {$A_2$};
%\draw (ZC) [above] node {$A_3$};

\draw (1.5,1) node {$A'_1$};
\draw (3.5,1) node {$A'_2$};
\draw (2.5,2.5) node {$A'_3$};

\draw (4.8,2.6) node {$S_\epsilon$};

%%%%%%%%%%%%%%%%%%%

\def\Ax_{6.5}
\def\Ay_{0.2}
\def\Bx_{11.0}
\def\By_{0.2}
\def\Cx_{8.75}
\def\Cy_{3.8}

\coordinate (A') at (\Ax_, \Ay_);
\coordinate (B') at (\Bx_, \By_);
\coordinate (C') at (\Cx_, \Cy_);

\coordinate (AB') at (0.5*\Ax_ + 0.5*\Bx_, 0.5*\Ay_ + 0.5*\By_);
\coordinate (AC') at (0.5*\Ax_ + 0.5*\Cx_, 0.5*\Ay_ + 0.5*\Cy_);
\coordinate (BC') at (0.5*\Cx_ + 0.5*\Bx_, 0.5*\Cy_ + 0.5*\By_);

\coordinate (Z') at (0.333*\Ax_ + 0.333*\Bx_ + 0.333*\Cx_, 0.333*\Ay_ + 0.333*\By_ + 0.333*\Cy_);

\filldraw [cyan, opacity=1] (A') -- (AB') -- (Z') -- (AC') -- cycle;
\filldraw [green, opacity=1] (B') -- (AB') -- (Z') -- (BC') -- cycle;
\filldraw [red, opacity=1] (C') -- (AC') -- (Z') -- (BC') -- cycle;

\draw (Z') -- (AB');
\draw (Z') -- (AC');
\draw (Z') -- (BC');

\color{black}
\draw [thick] (A') -- (B') -- (C') -- cycle;

\color{black}
%vertices
\fill (A') circle (2pt);
\fill (B') circle (2pt);
\fill (C') circle (2pt);

%labeling
\draw (A') [left] node {$\be_1$};
\draw (B') [right]  node {$\be_2$};
\draw (C') [right] node {$\be_3$};

\draw (8,1) node {$A''_1$};
\draw (9.5,1) node {$A''_2$};
\draw (8.75,2.5) node {$A''_3$};

\end{tikzpicture}

\end{figure}
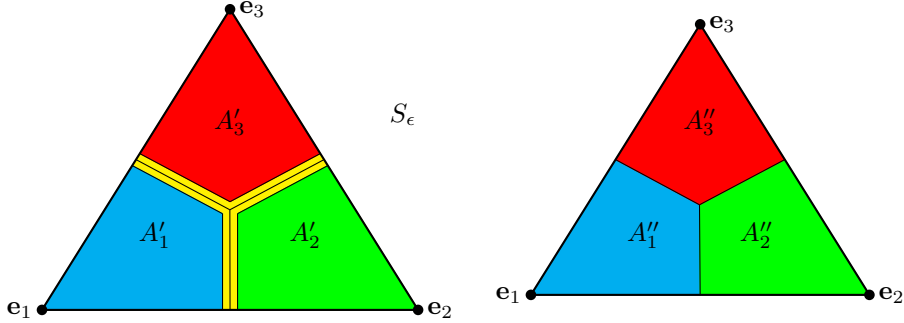

Next, we set $\epsilon' = \epsilon \sqrt{2}$ and define
$$ A''_i = A'_i - \epsilon' \be_i = \{ \bx - \epsilon' \be_i: \bx \in A'_i \}.$$
Thus $A''_i$ is a shifted copy of $A'_i$, where we push $A'_i$ slightly away from vertex $\be_i$. The sets $A''_i$ live in the hyperplane $\sum_{i=1}^{k} x_i = 1-\epsilon'$ rather than $\sum_{i=1}^{k} x_i = 1$. We claim that the sets $A''_i$ are still disjoint: Suppose that $A''_i \cap A''_j = (A'_i - \epsilon' \be_i) \cap (A'_j - \epsilon' \be_j) \neq \emptyset$. This would mean that there are points $\bx \in A'_i, \by \in A'_j$ such that $\bx - \epsilon' \be_i = \by - \epsilon' \be_j$. In other words, $\|\bx-\by\| = \epsilon' \|\be_i-\be_j\| =  \epsilon' \sqrt{2} = 2 \epsilon$. Take the midpoint $\frac12(\bx + \by)$: this point is in the simplex $\Delta_k$ (by convexity), and hence it is in some set $A_\ell$, where either $\ell \neq i$ or $\ell \neq j$ (possibly both). Assume without loss of generality that $\ell \neq i$. Then by the closedness of $A_i$ and $A_\ell$, between $\bx$ and $\frac12 (\bx+\by)$ there exists a point $\bx' \in A_i \cap A_\ell$. We get a contradiction, because $\| \bx' - \bx \| \leq \epsilon$ and so $\bx$ would not be included in $A'_i$.

We also observe that $A''_i \subseteq (1-\epsilon') \cdot \Delta_k = \{ \bx \geq 0: \sum_{i=1}^{k} x_i = 1-\epsilon'\}$. 
This is because for every $\bx \in A''_i$, we have $\by \in A'_i$ such that $\bx = \by - \epsilon' \be_i$. By assumption, $y_i \geq \epsilon_0 > \epsilon'$, and $\by \in \Delta_k$. Therefore $x_i = y_i - \epsilon' > 0$ and $\sum_{i=1}^{k} x_i = 1 - \epsilon'$. 
We conclude that $A''_1,\ldots,A''_k$ are disjoint subsets of $(1-\epsilon') \cdot \Delta_k$,
obtained by an isometry from $A'_1,\ldots,A'_k$ and therefore
$$ \sum_{i=1}^{k} \mu_{k-1}(A'_i) = \sum_{i=1}^{k} \mu_{k-1}(A''_i) \leq (1-\epsilon')^{k-1} \mu_{k-1}(\Delta_k).$$
Recall that $A'_1,\ldots,A'_k$ are also disjoint and $\bigcup_{i=1}^{k} A'_i = \Delta_k \setminus S_\epsilon$.
Therefore,
$$ \mu_{k-1}(S_\epsilon) = \mu_{k-1}(\Delta_k) - \sum_{i=1}^{k} \mu_{k-1}(A'_i)
 \geq \left(1 - (1-\epsilon')^{k-1} \right) \mu_{k-1}(\Delta_k). $$
By the definition of Minkowski content, we have
$$ \mu_{k-2}^-\left(S \right) = \liminf_{\epsilon \rightarrow 0^+} \frac{\mu_{k-1}(S_\epsilon)}{2\epsilon} \geq \liminf_{\epsilon \rightarrow 0^+} \frac{1 - (1-\epsilon')^{k-1}}{2\epsilon} \mu_{k-1}(\Delta_k) $$
$$ = \lim_{\epsilon \rightarrow 0^+} \frac{1 - (1-\epsilon \sqrt{2})^{k-1}}{2\epsilon} \mu_{k-1}(\Delta_k) = \frac{k-1}{\sqrt{2}} \mu_{k-1}(\Delta_k).$$
\end{proof}

\paragraph{Alternative proof of optimality.}
Here we give an alternative proof that the Voronoi partition has a \sep~of minimum Minkowski content,
avoiding an explicit computation of its volume.

\begin{proof}[Proof of Lemma~\ref{lem:Voronoi}]
Let us consider the Voronoi partition $(A^*_1,\ldots,A^*_k)$ (the proof for a general $\bz$ is similar).
We argue that the proof of Lemma~\ref{lem:sep-bound} is tight for this partition. As in the proof of Lemma~\ref{lem:sep-bound}, we define
$S^* = \bigcup_{i \neq j} (A^*_i \cap A^*_j)$, $S^*_\epsilon = \{ \bx \in \Delta_k: \exists \by \in S, \| \bx-\by \| \leq \epsilon \}$,
$A'_i = A^*_i \setminus S^*_\epsilon$ and $A''_i = A'_i - \epsilon' \be_i$, $\epsilon' = \epsilon \sqrt{2}$.
In the case of the Voronoi partition, these sets are explicitly described as follows:
\begin{itemize}
\item $A^*_i = \{ \bx \in \Delta_k: x_i = \max_{\ell \in [k]} x_\ell \}$,
\item $S^* = \{ \bx \in \Delta_k: \exists i \neq j, x_i = x_j = \max_{\ell \in [k]} x_\ell \}$,
%\item $S_\epsilon = \{ \bx \in \Delta_k: \exists i \neq j, x_i \geq x_j \geq \max_{\ell \neq i,j} x_\ell, |x_i-x_j| \leq \epsilon \sqrt{2} \}$,
\item $A'_i = \{ \bx \in \Delta_k: x_i > \epsilon' + \max_{\ell \neq i} x_\ell \}$,
\item $A''_i = \{ \bx - \epsilon' \be_i: \bx \in \Delta_k, x_i > \epsilon' + \max_{\ell \neq i} x_\ell \}$.
\end{itemize}
The description of $A'_i$ is valid because for $\bx \in A^*_i$, it is possible to find a point in $S^*$ within distance $\epsilon$ of $\bx$ if and only if the maximum coordinate $x_i$ is within $\epsilon' = \epsilon \sqrt{2}$ of the second largest coordinate --- then we can replace the two largest coordinates by their average and obtain a point in $S^*$.
The description of $A''_i$ follows by definition.

Consider now the scaled-down simplex $(1-\epsilon') \cdot \Delta_k$. By the proof of Lemma~\ref{lem:sep-bound}, the sets $A''_i$ are disjoint subsets of $(1-\epsilon') \cdot \Delta_k$. We show that in this case, we actually have $\sum_{i=1}^{k} \mu_{k-1}(A''_i) = \mu_{k-1}((1-\epsilon') \Delta_k)$. This is because for any point $\bx' \in (1-\epsilon') \cdot \Delta_k$, if the maximum coordinate $x'_i$ of $\bx'$ is unique then $\bx = \bx' + \epsilon' \be_i$ is a point in $\Delta_k$ such that $x_i > \epsilon' + \max_{\ell \neq i} x_\ell$. Therefore, $\bx \in A'_i$ which implies that $\bx' \in A''_i$. The points $\bx' \in (1-\epsilon') \cdot \Delta_k$ whose maximum coordinate is not unique form a set of $(k-1)$-dimensional measure zero. Therefore, $(1-\epsilon') \Delta_k$ is covered by $\bigcup_{i=1}^{k} A''_i$ up to a set of measure zero, and $\sum_{i=1}^{k} \mu_{k-1}(A'_i) = \sum_{i=1}^{k} \mu_{k-1}(A''_i) = \mu_{k-1}((1-\epsilon') \Delta_k) = (1-\epsilon \sqrt{2})^{k-1} \mu_{k-1}(\Delta_k)$. We also have $S^*_\epsilon = \Delta_k \setminus \bigcup_{i=1}^{k} A'_i$. This shows that all the inequalities in the proof of Lemma~\ref{lem:sep-bound} are tight and the Minkowski content of the \sep~$S^*$ is exactly
$$ \mu_{k-2}(S^*) = \lim_{\epsilon \rightarrow 0^+} \frac{\mu_{k-1}(S^*_\epsilon)}{2\epsilon} 
 = \lim_{\epsilon \rightarrow 0^+} \frac{1 - (1-\epsilon \sqrt{2})^{k-1}}{2\epsilon} \mu_{k-1}(\Delta_k)
%= \lim_{\epsilon \rightarrow 0^+} \frac{1 - (1-\epsilon \sqrt{2})^{k-1}}{2\epsilon} \mu_{k-1}(\Delta_k)
 = \frac{k-1}{\sqrt{2}} \mu_{k-1}(\Delta_k).$$
\end{proof}

\section{Discussion and open questions}
\label{sec:discussion}

Sperner's Lemma extends to general polytopes in the following sense \cite{LoeraPS02}: 

\medskip \noindent
{\em For any coloring of a triangulation of a $d$-dimensional polytope with $n$ vertices by $n$ colors, such that each point on a face $F = \mbox{conv}(\{\bv_i: i \in A\})$ must be colored with a color in $A$, there are at least $n-d$ full-dimensional simplices with $d+1$ distinct colors.}

\medskip

It is natural ask whether our results also extend to general polytopes.

\paragraph{Possible extensions to polytopes.}

Consider the example of $P$ being a square (Figure~\ref{fig:square}). The Voronoi partition $(A_1,A_2,A_3,A_4)$ 
%is optimal with respect to the measure $\sum_{i<j} \| \bv_i - \bv_j \| \mu(A_i \cap A_j)$ (by Theorem~\ref{thm:polytope-partition}), but not with respect to the measure $\mu( \bigcup_{i \neq j} (A_i \cap A_j))$.
is not optimal with respect to the total length of the \sep.
The \sep~of the Voronoi partition has total length $2$, whereas total length arbitrarily close to $\sqrt{2}$ is achieved by the partition $(B_1,B_2,B_3,B_4)$.

\begin{figure}[h]

\caption{Two partitions of a square.}

\label{fig:square}

\begin{tikzpicture}

\node at (-1,0) {};

\def\Ax{0}
\def\Ay{0}
\def\Bx{4}
\def\By{0}
\def\Cx{4}
\def\Cy{4}
\def\Dx{0}
\def\Dy{4}

\coordinate (A) at (\Ax, \Ay);
\coordinate (B) at (\Bx, \By);
\coordinate (C) at (\Cx, \Cy);
\coordinate (D) at (\Dx, \Dy);

\coordinate (AB) at (0.5*\Ax + 0.5*\Bx, 0.5*\Ay + 0.5*\By);
\coordinate (BC) at (0.5*\Bx + 0.5*\Cx, 0.5*\By + 0.5*\Cy);
\coordinate (Z) at (0.5*\Ax + 0.5*\Cx, 0.5*\Ay + 0.5*\Cy);
\coordinate (CD) at (0.5*\Cx + 0.5*\Dx, 0.5*\Cy + 0.5*\Dy);
\coordinate (AD) at (0.5*\Ax + 0.5*\Dx, 0.5*\Ay + 0.5*\Dy);

\filldraw [cyan, opacity=1] (A) -- (AB) -- (Z) -- (AD) -- cycle;
\filldraw [green, opacity=1] (B) -- (AB) -- (Z) -- (BC) -- cycle;
\filldraw [red, opacity=1] (C) -- (BC) -- (Z) -- (CD) -- cycle;
\filldraw [orange, opacity=1] (D) -- (CD) -- (Z) -- (AD) -- cycle;

\draw (A) -- (B) -- (C) -- (D) -- cycle;

\draw (Z) -- (AB);
\draw (Z) -- (BC);
\draw (Z) -- (CD);
\draw (Z) -- (AD);

\node at (1,1) {$A_1$};
\node at (3,1) {$A_2$};
\node at (3,3) {$A_3$};
\node at (1,3) {$A_4$};

\color{black}
%vertices
\fill (A) circle (2pt);
\fill (B) circle (2pt);
\fill (C) circle (2pt);
\fill (D) circle (2pt);

%%%%%%%%%%%%%%%%%

\def\Ax_{6}
\def\Ay_{0}
\def\Bx_{10}
\def\By_{0}
\def\Cx_{10}
\def\Cy_{4}
\def\Dx_{6}
\def\Dy_{4}

\coordinate (A_) at (\Ax_, \Ay_);
\coordinate (B_) at (\Bx_, \By_);
\coordinate (C_) at (\Cx_, \Cy_);
\coordinate (D_) at (\Dx_, \Dy_);

\coordinate (AB_) at (0.9*\Ax_ + 0.1*\Bx_, 0.9*\Ay_ + 0.1*\By_);
\coordinate (BC_) at (0.1*\Bx_ + 0.9*\Cx_, 0.1*\By_ + 0.9*\Cy_);
\coordinate (CD_) at (0.9*\Cx_ + 0.1*\Dx_, 0.9*\Cy_ + 0.1*\Dy_);
\coordinate (AD_) at (0.9*\Ax_ + 0.1*\Dx_, 0.9*\Ay_ + 0.1*\Dy_);

\coordinate (AZ_) at (0.95*\Ax_ + 0.05*\Cx_, 0.95*\Ay_ + 0.05*\Cy_);
\coordinate (CZ_) at (0.05*\Ax_ + 0.95*\Cx_, 0.05*\Ay_ + 0.95*\Cy_);

\filldraw [cyan, opacity=1] (A_) -- (AB_) -- (AZ_) -- (AD_) -- cycle;
\filldraw [green, opacity=1] (B_) -- (AB_) -- (AZ_) -- (CZ_) -- (BC_) -- cycle;
\filldraw [red, opacity=1] (C_) -- (BC_) -- (CZ_) -- (CD_) -- cycle;
\filldraw [orange, opacity=1] (D_) -- (CD_) -- (CZ_) -- (AZ_) -- (AD_) -- cycle;

\draw (A_) -- (B_) -- (C_) -- (D_) -- cycle;

\draw (AB_) -- (AD_);
\draw (AZ_) -- (CZ_);
\draw (BC_) -- (CD_);

\node at (5.7,0) {$B_1$};
\node at (9,1) {$B_2$};
\node at (10.3,4) {$B_3$};
\node at (7,3) {$B_4$};

\color{black}
%vertices
\fill (A_) circle (2pt);
\fill (B_) circle (2pt);
\fill (C_) circle (2pt);
\fill (D_) circle (2pt);
\end{tikzpicture}

\end{figure}
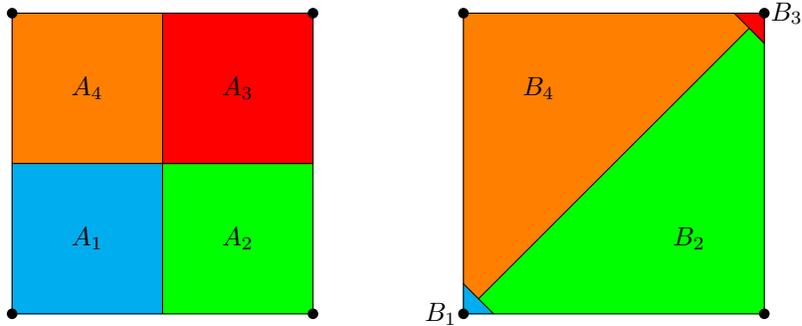

In general, we do not know what the partition minimizing $\mu(\bigcup_{i \neq j}(A_i \cap A_j))$ looks like,
even in the case of a non-regular simplex. 
We believe that the \sep~should still be polyhedral (piecewise linear) for an optimal Sperner-admissible partition of any polytope.

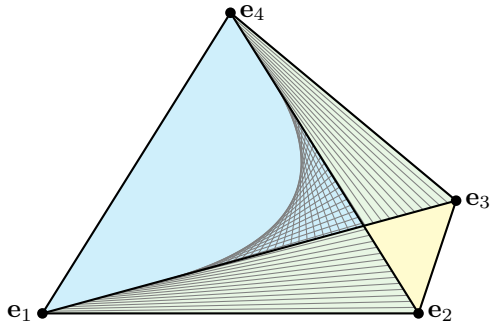
\begin{figure}[h]

\caption{An optimal partition between two pairs of faces of the tetrahedron.}

\label{fig:saddle}

\begin{tikzpicture}

\draw (-3,0) node {   };

%\draw [color=gray] (5,6.2) -- (0.5,-0.5) -- (9.5,-0.5) -- (5,6.2);
%\draw (1.0,-0.2) node {$V$};

\def\N{50}
\def\Ax{0}
\def\Ay{0}
\def\Bx{5}
\def\By{0}
\def\Cx{5.5}
\def\Cy{1.5}
\def\Dx{2.5}
\def\Dy{4}

\coordinate (A) at (\Ax, \Ay);
\coordinate (B) at (\Bx, \By);
\coordinate (C) at (\Cx, \Cy);
\coordinate (D) at (\Dx, \Dy);

\filldraw [cyan, opacity=0.1] (A) -- (C) -- (D) -- cycle;
\filldraw [cyan, opacity=0.1] (A) -- (B) -- (D) -- cycle;
\filldraw [yellow, opacity=0.1] (A) -- (B) -- (C) -- cycle;
\filldraw [yellow, opacity=0.1] (B) -- (C) -- (D) -- cycle;

\color{gray}
\foreach \i in {0,...,\N}
	{
	\coordinate (P) at ( \Ax + \i/\N*\Cx - \i/\N*\Ax, \Ay + \i/\N*\Cy - \i/\N*\Ay );
	\coordinate (Q) at ( \Bx + \i/\N*\Dx - \i/\N*\Bx, \By + \i/\N*\Dy - \i/\N*\By );
	\draw (P) -- (Q);
	}

\color{black}
\draw [thick] (A) -- (B) -- (C) -- (D) -- (A);
\draw [thick] (A) -- (C);
\draw [thick] (B) -- (D);

\color{black}
%vertices
\fill (A) circle (2pt);
\fill (B) circle (2pt);
\fill (C) circle (2pt);
\fill (D) circle (2pt);

%labeling
\draw (A) [left] node {$\be_1$};
\draw (B) [right]  node {$\be_2$};
\draw (C) [right] node {$\be_3$};
\draw (D) [right] node {$\be_4$};
\end{tikzpicture}

\end{figure}

We remark that depending on the coloring conditions on the surface of the polyhedron, the optimal \sep~may be non-linear: For a tetrahedron, the optimal partition that separates the pair of faces $\mbox{conv}(\be_1,\be_2,\be_3) \cup \mbox{conv}(\be_2,\be_3,\be_4)$ from $\mbox{conv}(\be_1,\be_2,\be_4) \cup \mbox{conv}(\be_1,\be_3,\be_4)$, is the minimal surface whose boundary is the non-planar 4-gon $\be_1$-$\be_2$-$\be_4$-$\be_3$. This is a saddle-shaped quadratic surface (see Figure~\ref{fig:saddle}).

\paragraph{Other open questions.}
%\label{sec:open-questions}

We have proved several results about colorings of the simplex. Our first result (Proposition~\ref{prop:EV}) can be viewed as being at the opposite end of the spectrum from Sperner's Lemma: Instead of the existence of a rainbow cell, we proved a lower bound on the number of non-monochromatic cells. Due to the motivating application of \cite{EV14}, we considered a special hypergraph embedded in the simplex rather than a full subdivision. A natural question is whether an analogous statement holds for simplicial subdivisions. 

More generally, we might ``interpolate" between Sperner's Lemma and our result, and ask: How many cells must contain at least $j$ colors? It is clear that these questions depend on the structure of the subdivision, and some assumption of regularity would be needed to obtain a general result.
Similarly, we may ask, for Sperner-admissible geometric partitions of the simplex, what is the minimum possible volume of the set where at least $j$ colors meet? Furthermore, as we discussed above, are there generalizations of these statements to other polytopes?

% More specifically, we can ask these questions about the concrete subdivision defined in Section~\ref{sec:HJALC}.
%\begin{itemize}
%\item For a Sperner-admissible labeling of a ``regular simplicial subdivision" (e.g., the one defined in Section~\ref{sec:???}), what is the minimum possible %number of non-monochromatic cells? What is the minimum possible number of cells containing at least $j$ colors?
%\end{itemize}
%We conjecture that for constant $j \leq k$ and $q \rightarrow \infty$, the number of cells containing at least $j$ colors is $\Omega(q^{k-j})$. 
%(For $j=k-1$, this can be shown by an argument similar to the proof of Sperner's Lemma.)
%We remark that while obtaining a bound of $\Omega(q^{k-2})$ would be relatively easy in the case of Proposition~\ref{prop:EV}, it is crucial for our application that we get the tight multiplicative constant as well.

Another question is, what is the Sperner-admissible labeling of the Simplex-Lattice Hypergraph $H_{k,q}$ (defined in Section~\ref{sec:prelims}) minimizing the maximum number of colors on a hyperedge? We have proved that $4$ colors suffice but it is possible that $2$ colors are enough (see Proposition~\ref{prop:4-colors}).
Is there a Sperner-admissible labeling of the hypergraph $H_{k,q}$, for sufficiently large $q$, such that each hyperedge uses at most $2$ colors?
%This would have a consequence for fair division as in Section~\ref{sec:fair-division}.
%We remark that such a labeling can be designed for $k=4$ and a sufficiently large $q$ (we omit the proof). 

Finally, we remark that Proposition~\ref{prop:4-colors} does not have a continuous counterpart for geometric partitions: As we discussed earlier, for any Sperner-admissible partition of a simplex there is a point where all the parts meet, by the Knaster-Kuratowski-Mazurkiewicz Lemma \cite{KKM29}.

\mypar{Acknowledgement.} {\em The second author is indebted in many ways to Jirka Matou\v{s}ek, who introduced him to Sperner's Lemma in an undergradute course at Charles University a long time ago.}

%%%%%%%%%%%%%%%%%%%%%%%%%%%%%%%%%%%%%%%%%%%%%%%%%%

\bibliographystyle{plain}

%%%%%%%%%%%%%%%%%%%%%%%%%%%%%%%%%%%%%%%%%%%%%%%%%%

\end{document}